\documentclass[12pt]{article}

\usepackage{subfigure}

\usepackage{latexsym}
\usepackage{graphics}
\usepackage{amsmath}
\usepackage{xspace}
\usepackage{amssymb}
\usepackage{psfrag}
\usepackage{epsfig}
\usepackage{amsthm}
\usepackage{pst-all}

\usepackage{color}

\definecolor{Red}{rgb}{1,0,0}
\definecolor{Blue}{rgb}{0,0,1}
\definecolor{Olive}{rgb}{0.41,0.55,0.13}
\definecolor{Yarok}{rgb}{0,0.5,0}
\definecolor{Green}{rgb}{0,1,0}
\definecolor{MGreen}{rgb}{0,0.8,0}
\definecolor{DGreen}{rgb}{0,0.55,0}
\definecolor{Yellow}{rgb}{1,1,0}
\definecolor{Cyan}{rgb}{0,1,1}
\definecolor{Magenta}{rgb}{1,0,1}
\definecolor{Orange}{rgb}{1,.5,0}
\definecolor{Violet}{rgb}{.5,0,.5}
\definecolor{Purple}{rgb}{.75,0,.25}
\definecolor{Brown}{rgb}{.75,.5,.25}
\definecolor{Grey}{rgb}{.5,.5,.5}

\newcommand{\R}{\mathbb{R}}
\newcommand{\Z}{\mathbb{Z}}

\newcommand{\C}{\mathbb{C}}

\newcommand{\ignore}[1]{\relax}

\newtheorem{theorem}{Theorem}[section]

\newtheorem{lemma}[theorem]{Lemma}

\newtheorem{coro}[theorem]{Corollary}

\definecolor{Red}{rgb}{1,0,0}
\definecolor{Blue}{rgb}{0,0,1}
\definecolor{Olive}{rgb}{0.41,0.55,0.13}
\definecolor{Green}{rgb}{0,1,0}
\definecolor{MGreen}{rgb}{0,0.8,0}
\definecolor{DGreen}{rgb}{0,0.55,0}
\definecolor{Yellow}{rgb}{1,1,0}
\definecolor{Cyan}{rgb}{0,1,1}
\definecolor{Magenta}{rgb}{1,0,1}
\definecolor{Orange}{rgb}{1,.5,0}
\definecolor{Violet}{rgb}{.5,0,.5}
\definecolor{Purple}{rgb}{.75,0,.25}
\definecolor{Brown}{rgb}{.75,.5,.25}
\definecolor{Grey}{rgb}{.5,.5,.5}
\definecolor{Pink}{rgb}{1,0,1}
\definecolor{DBrown}{rgb}{.5,.34,.16}
\definecolor{Black}{rgb}{0,0,0}

\usepackage{float}

\author{
{\sf David Gamarnik
\thanks{Operations research Center and Sloan School of Management, MIT. Email: gamarnik@mit.edu} 
\thanks{Support from the NSF grant DMS-2015517 is gratefully acknowledged.}
}
}

\begin{document}

\title{Correlation Decay and the Absence of Zeros Property of  Partition Functions}
\date{\today}

\maketitle

\begin{abstract}
Absence of (complex) zeros property is at the heart of the  interpolation method 
developed by Barvinok~\cite{barvinok2017combinatorics} for designing
deterministic approximation algorithms for various graph counting and computing partition functions problems. Earlier methods for solving the same problem include
the one based on the correlation decay property. Remarkably, the classes of graphs for which the two methods apply sometimes coincide 
or nearly coincide. In this paper we show that this is more than just a coincidence. We establish that if the  
interpolation method is valid for a family of graphs satisfying the self-reducibility property, 
then this family exhibits a form of correlation decay property which is asymptotic
Strong Spatial Mixing (SSM) at distances $\omega(\log n)$, where $n$ is the number of nodes of the graph. This applies in particular
to amenable graphs, such as graphs which are finite subsets of lattices. 

Our proof is based on a certain graph polynomial representation of  the associated partition function. This representation is
at  the heart of the design of the polynomial time algorithms underlying the
interpolation method itself. We conjecture that our result holds for all, and not just amenable graphs.
\end{abstract}


\section{Introduction}\label{section:Introduction}
The algorithmic question at the heart of this paper is one of designing a polynomial time algorithm for solving
various graph counting problems such as counting the number of independent sets of a graph, the number of proper colorings of a graph,
the number of partial matchings, etc. Generically, the problem is one of 
computing the partition function
$Z(G)$ associated with a graph induced possibly with some additional parameters such as the number of colors, list-colors, etc. 
As the existence of a polynomial
time algorithm for computing  partition functions amounts to the algorithmic complexity statement $P=\#P$, widely believed not to be true, 
the research has focused primarily on the question of designing algorithm for computing  partition functions 
approximately~\cite{barvinok2017combinatorics},\cite{JerrumSinclairHochbaumApproxAlgorithms},\cite{jerrum2003counting}.
The gold standard for the approximation algorithms is the existence of a Fully Polynomial Time Approximation Scheme (FPTAS). 
Randomized FPTAS (typically abbreviated as FPRAS) based on the Markov Chain Monte Carlo method
have been known for a while for a variety of such problems~\cite{JerrumSinclairHochbaumApproxAlgorithms},\cite{jerrum2003counting}. 

General methods for deterministic approximation algorithms have been developed much later. The algorithmic method based on the correlation decay 
property was introduced first in Bandyopadhyay and Gamarnik~\cite{BandyopadhyayGamarnikCountingConference},\cite{BandyopadhyayGamarnikCounting}. 
The method did not amount to the FPTAS as it was only leading
to approximation of the logarithm of the associated partition function, and only for graphs with diverging girth. A version of 
the correlation decay method which led to the FTPAS was invented in a breakthrough work of  Weitz~\cite{weitzCounting}
for the problem of counting the number of independent sets of a graph. A number of subsequent works extended the method to other graph 
counting problems~\cite{BayatiGamarnikKatzNairTetali},\cite{GamarnikKatz},\cite{li2013correlation},\cite{lu2013improved}.

The most recent 
progress towards constructing deterministic FPTAS for graph counting problems  is the development by 
Barvinok~\cite{barvinok2017combinatorics},\cite{barvinok2016computing},\cite{barvinok2015computing},\cite{barvinok2017computing},\cite{barvinok2019approximating} 
of an 
algorithmic method based on the Taylor approximation of the associated complex valued interpolated
partition function. Specifically, one designs a family of partition functions $Z(G(z))$ parametrized by complex value $z$ such that
when $z=1$, $Z(G(z))=Z(G)$ and when $z=0$ the associated partition function $Z(G(0))$ is trivially computable. One then considers
the Taylor approximation of the log-partition function $\log Z(G(z))$ and computes its first $m$-terms for $m$ which is typically logarithmic
in the number of nodes. This can be done by the brute force method in quasi-polynomial time $n^{O(\log n)}$, where $n$
is the number of nodes in a graph, but also in just polynomial time  $n^{O(1)}$ in bounded degree graphs, using
a certain graph polynomial representation of the partition function. This representation was 
 developed by Patel and Regts~\cite{patel2017deterministic}, and it  is at the core of the approach
of the present paper. Barvinok's interpolation  
method provably works provided that the model exhibits the ''zero-freeness'' property, namely the set of zeros of the interpolated
function $Z(G(z))$ is outside a connected region containing $0$ and $1$. 
Several families of graphs were the method is effective  either coincide or nearly coincide with the families of graph for which the correlation
decay based method applies. For other families of graphs no correlation decay counterparts are known or those which are known appear to work in a more
restricted setting. Examples of the former include the problem of counting partial matchings, 
see~\cite{barvinok2017combinatorics} and \cite{BayatiGamarnikKatzNairTetali},  where both methods apply to all bounded degree graphs, and
the problem of counting independent sets when the associated fugacity parameter satisfies
\begin{align}\label{eq:lambda_c}
\lambda< (d-1)^{d-1}/(d-2)^d,
\end{align} 
where $d$ is the largest degree of a graph.
The correlation decay based method in this regime was developed by 
Weitz in~\cite{weitzCounting}, 
and the interpolation method up to the  same threshold was developed by Peters and Regts~\cite{peters2019conjecture},
by establishing zero-freeness in this regime. 
Earlier slightly weaker bound for the zero-freeness property was established by 
Harvey,  Srivastava and Vondr{\'a}k~\cite{harvey2018computing}. 
At the same time for the problem of counting list-coloring of a graph, the correlation decay based
method was only developed for graphs with  list sizes  at least $2.58d+1$, as shown 
in Lu and Yin~\cite{lu2013improved},
whereas the polynomial interpolation method
applies under a significantly weaker assumption, where list sizes are at least roughly $1.764d$, as shown in
Liu, Sinclair and Srivastava~\cite{liu2019deterministic}. It is known that the correlation decay in the form
of the Strong Spatial Mixing (see below) does apply in this regime as well~\cite{gamarnik2015strong}, but turning
it into a counting algorithm is only known through the interpolation method, as was done in~\cite{liu2019deterministic}. 

What is the ultimate power and  limit of the interpolation method, and how are those related to the correlation decay property? 
We give a one-sided answer to this question by showing that the validity of the interpolation method for a self-reducible class
of graphs implies a form of correlation 
decay property we call the asymptotic Strong Spatial Mixing  (SSM). 
This is our main result stated as Corollary~\ref{coro:Main-corollary},
which is a simple implication of our main technical result stated in Theorem~\ref{theorem:main-result}. 
The self-reducibility refers to the property that the model remains in the family when some of the nodes have prescribed values.
We give several examples of such models, with independent set model on bounded degree graphs being one example, and list-coloring of a graph
problem being the second example. Our result is thus
stated for two types of interpolation schemes successfully used in the past, which we call Type I and Type II interpolations. The first
is one was used to design an FPTAS  for counting independent sets as in~\cite{patel2017deterministic}, and the second is a generalization of the type 
used for designing FPTAS for counting list-colorings of a graph, as in~\cite{liu2019deterministic}. Both interpolation types 
are defined precisely in the body of the paper.

We now discuss briefly the SSM  property. Its weaker counterpart, the Spatial Mixing (SM)
property, is a property which is stated in terms of the Gibbs distribution associated with the partition function.
The SM is widely studied in the statistical physics literature~\cite{GeorgyGibbsMeasure} and is directly related to the properties
of uniqueness of Gibbs measures on infinite graphs.  Roughly speaking, it is the property that
the marginal distribution with respect to the Gibbs measure associated with a subset of nodes of a graph is asymptotically 
independent from the conditioning of the boundary of a neighborhood of the set, when the radius of the neighborhood is sufficiently large.
Typically, such a decay of correlations is upper bounded by a function converging to zero as radius diverges to infinity, and this
function is uniform in the choice of the set and graph size itself.
The SSM is a strengthened version of the SM which is SM applied to the original graph being reduced by setting some subset of the nodes
of the graph to some fixed values, similarly to the self-reducibility property. 
The asymptotic version of the SSM property that we consider in this paper is a ''non-uniform'' version of the SSM
which occurs at radius values that depend on the graph choice. Specifically, we establish that the zero-freeness property implies
the SSM at radius value $\omega(\log n)$ where
$n$ is the cardinality of the node set. As such the property is applicable in particular to graphs, for which  for any fixed node the number of nodes
with distance  $\omega(\log n)$ from this node still constitutes the bulk of the graph. The special case includes
all subgraphs of lattices, and  amenable type graphs in general. However, it does not apply to graph sequences which are expanders,
and specifically the graphs where nodes beyond distance $\omega(\log n)$ from any given node simply might not exist. Whether our result
extends to the case of expanders is thus left  open.  
We note
that the SSM by itself does not render the partition function estimation algorithms and additional steps are needed such
as either the SSM on the associated self-avoiding tree as in~\cite{weitzCounting}, or the SSM on the associated computation tree
as in~\cite{GamarnikKatz}.

One could wonder whether the opposite implication is true as well.  Also one could wonder whether our result is implied
by some existing results in the non-algorithmic literature. 
For restricted models such as lattices, indeed the equivalence between the zero-freeness and long-range independence has been known
for a while as discussed in the classical works of Dobrushin and Shlossman~\cite{dobrushin1987completely}. Remarkably, however, such
an equivalence  does not
extend to unstructured graph sequences, such as for example sequences of all bounded degree graphs, and in fact the lack of zero-freeness can coexist
with long-range independence, as we now demonstrate. Indeed, consider any model which violates the 
zero-freeness property, for example the hard-core model which violates
condition (\ref{eq:lambda_c}) above. For this choice of $\lambda$ and $d$ consider any constant size graph with degree $d$ (for example a clique
on $d+1$ nodes) and a disjoint union of $n/(d+1)$ of such graphs. The set of zeros of the associated partition function is the set of zeros
of one individual clique and thus violates the zero-freeness property. Yet the model trivially exhibits the long-range independence for distances beyond $d$
including the SSM property.

Our result does not rule out the applicability of the interpolation method beyond the SSM regime if some modifications are introduced. For example,
Helmuth, Perkins and Regts~\cite{helmuth2020algorithmic} and Jenssen,  Keevash  and Perkins~\cite{jenssen2020algorithms} 
apply the method to low-temperature models on lattices and bi-partite graphs in general by taking the advantage of the simple
structure of ground states on these models and appropriate redefining of the underlying partition function. 

The fact the long-range dependence might indicate a barrier for a successful implementation of the interpolation argument should not be entirely surprising in light
of some of the hardness results implied by the long-range dependence. In particular, Sly~\cite{sly2010computational} 
has shown that for general graphs with degree at most $d$
no FPTAS exists for values $\lambda$ strictly violating the condition (\ref{eq:lambda_c}), unless $P=NP$. The argument leverages the fact that
bi-partite sparse random graphs exhibit a long-range dependence which can be then used as a gadget in a more complicated graph structure to argue
that the existence of an FPTAS for computing the partition function of this graph structure
implies an approximation algorithm for the MAX-CUT problem, which is known not to admit an approximation algorithm unless P=NP. 
The Sly's result by itself though does not imply our result, as our result  is not based on  any complexity-theoretic assumptions.

The proof of our result draws heavily on the work of Patel and Regts~\cite{patel2017deterministic}. It was shown in this paper that the 
interpolated partition function $Z(G(z))$ for many models can be written as the so-called graph polynomial, namely, a polynomial 
with coefficients expressed in terms of linear combination of subgraph counts. It is then shown that the coefficients of the Taylor expansion
of $\log Z(G(z))$ can be expressed entirely in terms of counts of \emph{connected} subgraphs. This was used crucially to ensure the  polynomiality 
as opposed to just quasi-polynomiality of the running  time of the algorithms. The key fact used in this approach
was the fact that  counting the number of connected graphs of order $O(\log n)$ nodes on bounded
degree graphs can be done in time $n^{O(1)}$ as opposed to quasi-polynomial time $n^{O(\log n)}$. For us, though, this property has a completely
different ramification. The conditional marginal distribution of any set $S$, when  conditioning  on the boundary $\partial B(S,R)$
 of an $R$-radius neighborhood $B(S,R)$ of $S$
can be written as a ratio of partition functions of the original and reduced models, using the self-reducibility of the class of models we consider. 
The success of the interpolation argument with up to $O(\log n)$ terms of the Taylor approximation implies that the Taylor approximation 
of this ratio involves
only connected graphs of order $O(\log n)$ which ''touch'' either the set $S$, or the boundary $\partial B(S,R)$ or both. But if $R$ is $\omega(\log n)$, 
no connected graph of size at most  $O(\log n)$ can touch both $S$ and $\partial B(S,R)$. This implies that the Taylor 
approximations of the conditional
marginal distribution, which we call the conditional ''pseudo-marginal'', have the same value as the unconditional pseudo-marginal values, thus implying
the long-range independence at distance $\omega(\log n)$. Our argument as implemented in the current version does not seem to be capable of showing
the long-range independence at distances  $c\log n$ for small constants $c$, which would be needed to extend 
our result to a larger families of graphs, including expanders.

The remainder of the paper is structured as follows. The model definition and the review of the interpolation method are subject of the next
section. In the same section we overview some examples and introduce the definition of pseudo-marginal distributions. 
The definition of the SSM and the asymptotic SSM, and the statements of the main results are in 
Section~\ref{section:main-result}. Some preliminary technical results are in Section~\ref{section:preliminary-results}.
The proof of the main result is found in Section~\ref{section:main-result-proof}.

We close this section with some notational convention. For every integer $K$, $[K]$ denotes the set $1,\ldots,K$. This will
be typically used as the set of colors in this paper.
For every graph $H$, we write $V(H)$ and $E(H)$ for the set of nodes
and the set of edges of $H$, respectively. Two graphs $H_1$ and $H_2$ are disjoint if $V(H_1)\cap V(H_2)=\emptyset$.
By default this means that there are no edges with one end in $V(H_1)$ and another end in $V(H_2)$.
Given a graph
$G=(V,E)$ and node $u\in V$, $B(u)$ denotes the set of neighbors 
of $u$, that is the set of nodes $v\in V$ such that $(u,v)\in E$. For each integer $R$, $B(u,R)$ denotes the set of nodes
$v$ accessible from $u$ via paths of length at most $R$. In particular $B(u,1)=B(u)$. Let $\partial B(u,R)=B(u,R)\setminus B(u,R-1)$
be the set of ''boundary nodes'' -- nodes at distance precisely $R$ from $u$.
 The distance $d(u,v)$ between nodes $u$ and $v$
is the length of a shortest path connecting $u$ to $v$. Namely $d(u,v)=\min t$ such that there exists nodes $u_0=u,u_1,\ldots,u_t=v$
such that each pair $(u_i,u_{i+1}), 0\le i\le t-1$ is an edge. Similarly, for any set $S\subset V$, $B(S,R)=\cup_{u\in S} B(u,R)$
and $\partial B(S,R)=B(S,R)\setminus B(S,R-1)$.
The degree of the graph is $\max_u |B(u)|$. A graph $H$ is connected if $\cup_{R\ge 1}B(u,R)=V(H)$ for each node $u\in V(H)$.
A graph is disconnected if it is not connected.

\section{Graph homomorphisms and the  interpolation method}
Suppose $G=(V,E)$ is a simple undirected graph on the node set $V=V(G)$ and the edge set $E=E(G)$. Given a positive integer $K$, suppose
a vector $a^u\in \R_+^K$
with non-negative entries is associated with every node $u\in V$ of $G$, 
and a symmetric matrix $A^{(u,v)}\in \R_+^{K\times K}$ also with non-negative entries is associated with every edge $(u,v)\in E$ of $G$. 
Let $\mathcal{A}$ be short-hand notation for the collection $a^u, u\in V,  A^{(u,v)}, (u,v)\in E$. 
We will often refer to the elements of $[K]$
as colors and call the collection $\mathcal{A}$  list-coloring
of $G$ for reasons to be discussed below. 
Define
\begin{align}\label{eq:partition-function}
Z(G,\mathcal{A})\triangleq \sum_{\phi:V\to [K]}\prod_{u\in V}a_{\phi(u)}^u\prod_{(u,v)\in E}A_{\phi(u),\phi(v)}^{(u,v)}.
\end{align}
For any $\phi:V\to [K]$ letting
\begin{align}\label{eq:weight-w}
w(\phi)=\prod_{u\in V}a_{\phi(u)}^u\prod_{(u,v)\in E}A_{\phi(u),\phi(v)}^{(u,v)},
\end{align}
we have $Z(G,\mathcal{A})=\sum_{\phi:V\to [K]} w(\phi)$. 
We call this value the ''number'' of homomorphisms from $G$ to the collection $\mathcal{A}$. The justification for this definition is the special case when 
$a^u$ is the vector of ones for all $u$ and $A^{(u,v)}=A$ are edge independent with 
$A_{i,j}\in \{0,1\}$ for all $1\le i,j\le K$. 
In this case we can think of $A$ as an adjacency matrix of a 
 graph $H$
on $K$ nodes. This graph $H$ is allowed to have loops if some of $A_{i,i}$ equal to one.
Then $Z(G,\mathcal{A})$ counts the number of homomorphisms from $G$ into $H$, namely the number of maps $\phi:V\to V(H)$ such that for every $(u,v)\in E$ it 
is the case that also $(\phi(u),\phi(v))\in E(H)$. 

Throughout the paper we will be considering graphs $G$ associated with some  list-coloring $\mathcal{A}$, so we will use
a shorthand notation $G$ for a graph along with list-coloring. Thus $G$ is a triplet $(V,E,\mathcal{A})$ and we call $G$ a decorated graph. 
We use $Z(G)$ in place
of $Z(G,\mathcal{A})$ light of this notational change.

$Z(G)$ is also
called the partition function, a term more commonly used in the statistical physics literature. The partition functions naturally factorize over 
disjoint unions graphs. Namely, suppose $G_1=(V_j,E_j,\mathcal{A}_j)$ are two disjoint graphs. 
Let $G$ be the union of $G_1$ and $G_2$ with naturally associated union $\mathcal{A}$ of color-lists  $\mathcal{A}_1$ and $\mathcal{A}_2$. 
Then 
\begin{align}\label{eq:partition-factorizes}
Z(G)=Z(G_1)Z(G_2).
\end{align}
Let $\mathcal{G}$ denote the set of all decorated graphs $(V,E,\mathcal{A})$. 
The set  $\mathcal{G}$ is uncountable. Yet we will use the notation of the form $\sum_{H\in \mathcal{G}}\cdot$,
which will be well defined when only finitely many terms to be summed are non-zero.
For every positive integer $i$ let $\mathcal{G}_i\subset\mathcal{G}$ denote the (uncountable) set of all $i$-node decorated graphs $G=(V,E,\mathcal{A})$. Namely
$|V|=i$ for each such graph.  Let $\bar{\mathcal{G}}_i=\cup_{j\le i}\mathcal{G}_j$.
Denote by $\mathcal{G}_{i,\text{conn}}$ the subset of $\mathcal{G}_i$ consisting of only connected graphs. Let
$\bar{\mathcal{G}}_{i,\text{conn}}=\cup_{j\le i}\mathcal{G}_{i,\text{conn}}$.

Similarly, let $\mathcal{G}_{i,\text{edge}}$ be the uncountable set of all graphs which ares spanned by  $i$-edges ($|E|=i$). 
Namely, $(V,E,\mathcal{A})\in \mathcal{G}_{i,\text{edge}}$ if there exists a subset of edges $E'\subset E, |E'|=i$ such 
that the set of nodes incident to edges in $E'$ is the entire set $V$. We note that the same graph may belong to sets
$\mathcal{G}_{i,\text{edge}}$ with different values of $i$ as clearly  subsets of edges of different cardinality can span the same set of nodes.
The sets $\bar{\mathcal{G}}_{i,\text{edge}}, \mathcal{G}_{i,\text{edge},\text{conn}}$ and $\bar{\mathcal{G}}_{i,\text{edge},\text{conn}}$
are defined similarly.

Given a graph $G=(V,E,\mathcal{A})$, we now introduce the associated Gibbs measure $\mu$ on the set of mappings $\phi:V\to [K]$.
The measure is defined as follows: the probability
weight $\mu(\phi)$ associated with $\phi$ is $\mu(\phi)=w(\phi)/Z(G)\ge 0$. The measure is well defined only when 
$Z(G)$ is strictly positive. 
Clearly $\sum_\phi \mu(\phi)=1$, that is $\mu$ is indeed a probability measure.

Associated with Gibbs measure $\mu$ are marginal probability distributions for each subset of nodes $S\subset V$. 
Specifically, for any $S\subset V$ and any $\sigma\in [K]^S$ encoding a coloring assignment $\sigma:S\to [K]$, 
the associated marginal probability denoted by $\mu(G,S,\sigma)$ is
\begin{align}\label{eq:marginal}
\mu(G,S,\sigma)=Z^{-1}(G)\sum_{\phi: \phi(u)=\sigma(u), \forall u\in S}w(\phi).
\end{align}
Namely $\mu(G,S,\sigma)$ is simply the likelihood that  $\phi$ generated at random according to $\mu$, maps each $u\in S$ into $\sigma(u)$. 
Naturally, by the total probability law $\sum_{\sigma:S\to [K]} \mu(S,\sigma)=1$. 

Given two sets $S,T\subset V$ and colorings $\sigma:S\to [K],\tau:T\to [K]$ we will also write $\mu(G,S,\sigma|T,\tau)$ for the conditional
probability of the event $\phi(u)=\sigma(u), \forall u\in S$ when conditioned on the event $\phi(v)=\tau(v), \forall v\in T$. Thus
\begin{align*}
\mu(G,S,\sigma|T,\tau)={\mu(G,S\cup T,\sigma\cup\tau)\over \mu(G,T,\tau)},
\end{align*}
where $\sigma\cup\tau$ denotes the implied coloring of the union $S\cup T$. This is non-zero  only when $\sigma$ and $\tau$ are
consistent on the intersection $S\cap T$.

Next we observe that marginals $\mu(G,S,\sigma)$ can be conveniently written in terms of ratio of partition functions associated
with the original and the reduced  model, exhibiting the fundamental 
property of self-reducibility of our graph homomorphism model. 
Specifically, given $S\subset [V]$ and $\sigma:S\to [K]$,
let $\mathcal{A}_{S,\sigma}$ be the modified decoration of $G$ defined by the same values associated with node $a^{S,\sigma,u}=a^u, u\in V$ and
\begin{align}\label{eq:ASsigma}
A^{S,\sigma;(u,v)}_{i,j}=A^{(u,v)}_{i,j}\textbf{1}(i=\sigma(u)),
\end{align}
for every $(u,v)$ such that $u\in S$, and $A^{S,\sigma;(u,v)}_{i,j}=A^{(u,v)}_{i,j}$ if both $u,v\in V\setminus S$.
By symmetry this also means 
\begin{align*}
A^{S,\sigma;(u,v)}_{i,j}=A^{(u,v)}_{i,j}\textbf{1}(j=\sigma(v)),
\end{align*}
for every $(u,v)$ such that $v\in S$. In particular, the weight  of $\phi:V\to [K]$ according to the modified list is zero if
$\phi(u)\ne \sigma(u)$ for at least one $u\in S$, and it is $w(\phi)$ otherwise. Considering the partition function $Z(G_{S,\sigma})$
of the modified decorated graph $G_{S,\sigma}\triangleq (V,E, \mathcal{A}_{S,\sigma})$,
we obtain  the identity
\begin{align}\label{eq:marginal-ratio-part-f}
\mu(G,S,\sigma)={Z(G_{S,\sigma})\over Z(G)}.
\end{align}
Similarly, for any $S,\sigma:S\to [K],T,\tau:T\to [K]$,
\begin{align}\label{eq:marginal-ratio-part-f-conditional}
\mu(G,S,\sigma|T,\tau)={Z(G_{S\cup T,\sigma\cup\tau})\over Z(G_{T,\tau})},
\end{align}
with term $Z(G)$ cancelled out. We thus note that by definition $\mu(G,S,\sigma|T,\tau)=\mu(G_{T,\tau},S,\sigma)$.

While it would be arguably more natural to modify the decoration $\mathcal{A}$ by modifying node values to 
$a^{S,\sigma;u}_i=a^u_i\textbf{1}(i=\sigma(u))$,
the choice above is dictated by the  interpolation construction to be introduced below associated with the list-coloring problem.

We now discuss some common examples of the model above.

\subsection{Examples}\label{subsection:examples}

\subsubsection*{Independent Sets/Hard-core model}
An independent set of a graph $G$ is a subset $I\subset V$ of nodes which spans no edges. Namely $(u,v)\notin E$ for all $u,v\in I$. 
Fix a parameter $\lambda>0$, which is sometimes called fugacity in the statistical physics literature. 
The counting object of interest is $Z(G)\triangleq \sum_I \lambda^{|I|}$, 
where the sum is over all independent sets of $G$.
When $\lambda=1$, this is simply the total number of independent sets of the graph $G$. Letting $i_k(G)$ stand for the number of independent
sets of $G$ with cardinality $k$ and interpreting $i_0(G)$ as $1$, we also have
\begin{align*}
Z(G)=\sum_{0\le k\le |V|}i_k(G)\lambda^k.
\end{align*}

The model above is a special case of homomorphism counting given by $K=2$, and $\mathcal{A}$ given by $a^u=(1,\lambda)$ for all $u\in V$, 
$A_{2,2}^{(u,v)}=0$ and  $A_{i,j}^{(u,v)}=1$ for all other $1\le i,j\le 2$, for all edges $(u,v)\in E$.
Indeed, for any $\phi:V\to \{1,2\}$ such that $w(\phi)>0$, the set $I=\{u:\phi(u)=2\}$ is an independent set, since otherwise
having $(u,v)\in E$ for some $u,v\in I$ implies $A_{\phi(u),\phi(v)}=A_{2,2}=0$, namely $w(\phi)=0$. Also for every independent set $I$ and the associated
map $\phi(u)=2, u\in I, \phi(u)=1, u\notin I$, we have $w(\phi)=\lambda^{|I|}$. Thus indeed this model is a special case of the model 
(\ref{eq:partition-function}).

We note that the restrictions of the form $G\to G_{S,\sigma}$   does not change the model it in any meaningful way. Specifically, 
consider the reduced graph $\tilde G$ obtained by deleting from $G$ 
all nodes $u\in S$ 
such that $\sigma(u)=1$, and deleting all nodes $u$ and the associated 
neighborhoods $B(u)$, for nodes $u\in S$ such that $\sigma(u)=2$. In other words, $\tilde G$ is obtained by deleting all nodes which are forced not to belong
to an independent set by $\sigma$, and deleting all nodes which are actually forced to belong to an independent set by $\sigma$
along with their neighbors. 
Then $Z(G_{S,\sigma})=\lambda^k Z(\tilde G)$ where $k$ is the number of nodes $u\in S$ with  $\sigma(u)=2$ which are forced to be a part of an independent set.

\subsubsection*{Proper Colorings and Proper List-Colorings models}
For any positive integer $K$, let $a^u$ be the $K$-vector of ones for all nodes $u$, 
and let $A^{(u,v)}=A$ be edge independent and given by  $A_{i,j}=1$ when $1\le i\ne j\le K$ and $A_{i,i}=0, i=1,2,\ldots,K$.
Then for any $\phi:V\to [K]$, $w(\phi)=1$ when  the values $\phi(u)$ and $\phi(v)$ are distinct for all edges $(u,v)\in E$, and $w(\phi)=0$
otherwise. Namely, $w(\phi)=1$ iff $\phi$ corresponds to a proper coloring of $G$ with colors $1,2,\ldots,K$, and $Z(G)$ is the total number
of proper colorings of $G$. 

Turning to the list-coloring problem, suppose each node $u$ is associated with a list of colors $C(u)\subset [K]$.
A mapping $\phi:V\to [K]$
is a proper list-coloring if in addition to the requirement $\phi(u)\ne \phi(v)$ for each each $(u,v)\in E$ it is also the case that
$\phi(u)\in C(u)$ for each node $u$. This is again a special case of our model given by the following $\mathcal{A}$.
We let again $a^u$ be the vector of ones for all $u$, and let 
\begin{align*}
A^{(u,v)}_{i,j}=\textbf{1}(i\ne j, i\in C(u),j\in C(v)), \qquad \forall~(u,v)\in E.
\end{align*}
The number of proper list-colorings is then simply $Z(G)$ as defined per (\ref{eq:partition-function}).

\subsubsection*{Ising model}
Fix $K=2$ and $h,\beta>0$. Suppose $a=(1,e^h), A_{1,1}=A_{2,2}=e^\beta, A_{1,2}=A_{2,1}=e^{-\beta}$. Then
\begin{align*}
Z(G)=\sum_{\phi:V\to \{1,2\}}\exp\left(h\sum_{u\in V}(2\phi(u)-3)+\beta\sum_{(u,v)\in E}(2\phi(u)-3)(2\phi(v)-3)\right)
\end{align*}
The parameter $h$ is called the strength of the associated magnetic field and the parameter $\beta$ is called inverse temperature. A more canonical equivalent
way to represent this model is in terms of spin assignments $\sigma:V\to \{-1,1\}$, in which case $Z(G)$ is simply
\begin{align*}
\sum_{\sigma:V\to \{-1,1\}}\exp(h\sum_u\sigma(u)+\beta\sum_{u,v}\sigma(u)\sigma(v)).
\end{align*}
The equivalence is immediate by transformation $2\phi-3$ mapping $1$ and $2$ to $-1$ and $1$, respectively. The cases $\beta>0$, (respectively $\beta<0)$
is called ferromagnetic (respectively anti-ferromagnetic) Ising model. The model is interesting including  the case of no magnetic field  $h=0$.

\subsection{Interpolation method}
The key idea underlying the  interpolation method for computing partition functions $Z(G)$  relies on first replacing
the target decorated graph $G=(V,E,\mathcal{A})$, for which $Z(G)$ is hard to compute, by an alternative decoration 
$\hat{\mathcal{A}}$ on the same ground graph $(V,E)$,  for which the partition function
$Z(\hat G)$ can be easily  evaluated, where $\hat G=(V,E,\hat{\mathcal{A}})$. 
Then one builds a convenient interpolation $\mathcal{A}(z)$ 
between $\mathcal{A}$ and $\hat{\mathcal{A}}$, parametrized by some
complex parameter $z\in\C$ (with understanding that $a^u$ and $A^{(u,v)}$ are now complex valued), 
and  rewrites $\log Z(G)$ as $z$-variable Taylor expansion  around easy to compute $\log Z(\hat G)$.
One then computes  the polynomial associated with the Taylor expansion 
truncated at a sufficiently low degree terms and uses it to approximate $Z(G)$. The method works provided that the
partition function of the interpolated model as a function of $z$ 
is zero-free in the region containing the set of interpolating values of $z$, see~\cite{barvinok2017combinatorics} for the textbook
exposition of the method.

The main result in this paper concerns two types of interpolation schemes  which have been successfully used in some of the  earlier results. 
The first one concerns the independent set model and the second one concerns the proper list-coloring model. 
While there are other successful examples of  interpolation schemes,
we will focus on just these two  to illustrate the main ideas.

The first  interpolation type is motivated and easy to describe in terms of the problem of counting independent sets (hard-core model). 
Given $G$ and $\lambda>0$, introduce the following $z$-variable polynomial 
\begin{align}\label{eq:Partition-hard-core-interpolated}
Z(G(z))=\sum_I z^{|I|}\lambda^{|I|}=\sum_{0\le k\le |V|}i_k(G)z^k\lambda^k.
\end{align}
where the first sum is again over all independent sets $I$ of $G$. We see that $Z(G(z))$ is the partition function
of the model $G(z)=(V,E,\mathcal{A}(z))$ where $\mathcal{A}(z)$ is obtained from $\mathcal{A}$ simply by replacing $\lambda$ with $\lambda z$. 
 Trivially, $Z(G(0))=1$ and $Z(G(1))=Z(G)$. 
Let $f(z)=\log Z(G(z))$ (with the branch of logarithm appropriately fixed). Consider the infinite Taylor's expansion around $z=0$: 
\begin{align*}
f(z)=\sum_{k\ge 0}{1\over k!}f^{(k)}(0)z^k,
\end{align*}
where $f^{(k)}$ is the $k$-th order derivative of $f$. The idea of the  interpolation method is  that for 
$m$ small enough, typically logarithmic in $|V|$, the truncated expansion  
\begin{align}\label{eq:T_mGz}
T_m(G,z)\triangleq \sum_{0\le k\le m}{1\over k!}f^{(k)}(0)z^k
\end{align}
is a good approximation of $f$ in a connected region of $\C$ containing $0$ and $1$, provided $f(z)$ is substantially
distinct from zero in this region (zero-freeness).
Specifically, one proves that for any $\epsilon>0$ there exists $C$ such that if $m=C\log |V|$
then
\begin{align*}
1-\epsilon\le {\exp(T_m(G,1))\over Z(G)}\le 1+\epsilon.
\end{align*}
One then proceeds to establishing this zero-freeness property using various properties of the graph 
such as degree boundedness. This scheme has been implemented in~\cite{patel2017deterministic} where the zero-freeness was shown
for $\lambda$ satisfying (\ref{eq:lambda_c}) for graphs with degree at most $d$. 

As it turns out it is a tractable problem to compute the derivatives $f^{(k)}(0)$ in quasi-polynomial time 
for graphs with  degree bounded by some constant $\Delta$. As an explanation,
observe that the $k$-derivative $Z^{(k)}(G,0)$ of $Z(G,z)$ at $z=0$ is simply $k!i_k(G)$. When $k=O(\log |V|)$, $i_k(G)$  can be computed 
in quasi-polynomial time 
by brute-force method in time $|V|^{O(\log |V|)}$. Then one observes that 
the $k$-th derivative $f^{(k)}$ at $z=0$ can be expressed in a recursive way 
as sum-product of terms $Z^{(\ell)}(G,0), \ell\le k$, namely the sum-product of 
terms $i_\ell(G), \ell\le k$, thus allowing for a quasi-polynomial computation of $T_m(G,z)$ at any $z$. Setting $z=1$ one uses
$T_m(G,1)$ as an approximation of $Z(G,1)=Z(G)$. Importantly, the quasi-polynomiality can be improved to just
polynomiality using a clever  method based on representing partition function as graph polynomials of connected subgraph, 
as achieved in~\cite{patel2017deterministic}, and reducing the problem
to counting over connected subgraphs only. The key ideas behind this method are in fact used in our paper for establishing
the connection between the interpolation method and the correlation decay, and are represented in 
Lemmas~\ref{lemma:Zero-disconnected-uncolored}, \ref{lemma:Z-power-series} and~\ref{lemma:Zero-disconnected} below.
In particular, the graph polynomial representation allows one to express the approximate marginal probabilities (pseudo-marginals
to be defined below) in terms connected small subgraphs forcing such pseudo-marginals to have independence  over well-separated sets. 
 In the end, $\exp(T_m(G,z))$ evaluated at $z=1$ amounts to a deterministic FPTAS for approximation of $Z(G)$ up to any constant level of precision
$\epsilon$. In fact one can reach accuracy $\epsilon$ which is inverse polynomial in $|V|$: $\epsilon=n^{-\Omega(1)}$ by selecting the constant 
$C$ in $m=C\log n$ value appropriately large. 

The interpolation construction above concerning the independent set models will be referred to as Type I interpolation scheme below. 
It is only defined for the independent set model.

We now turn to the Type II interpolation model, which concerns models generalizing the proper list-coloring model. 
Given a (decorated) graph $G=(V,E,\mathcal{A})$,  we construct the modified $z$-dependent color-list $\mathcal{A}(z)$ as follows:
 $a^u(z)=a^u$ for all $z$, and  $A^{(u,v)}(z)$ is given by 
\begin{align*}
A^{(u,v)}(z)=J+(A^{(u,v)}-J)z,
\end{align*}
where $J$ is the $K\times K$ matrix of ones. We denote by $G(z)$ the triplet $(V,E,\mathcal{A}(z))$. 
When $z=1$ we have $Z(G(z))=Z(G)$, and when $z=0$,
$Z(G(z))$ trivializes to
\begin{align}\label{eq:LGA}
\prod_{u\in V}\left(\sum_{1\le i\le k}a_i^u\right)\triangleq L(G).
\end{align}
Then we again let $f(z)=\log Z(G(z))$ and define $T_m(G,z)$ by (\ref{eq:T_mGz}). We see that in the special case of the list-coloring problem,
\begin{align*}
Z(G(z))=\sum_{\phi:V\to [K]}z^{e(\phi)},
\end{align*}
where $e(\phi)$ is the total number of ''color violations'' of $\phi$. Namely the total number of nodes $u$ with $\phi(u)\notin C(u)$
and the total number of edges $(u,v)$ with $\phi(u)=\phi(v)$. This interpolation scheme was considered in~\cite{patel2017deterministic}
and~\cite{liu2019deterministic} with the latter leading to the deterministic FPTAS for the counting list-colorings problem.

\subsection{Pseudo-marginals}
If $T_m$ is a good approximation of the log-partition function with a well-controlled error, 
then it stands to reason that marginal distributions $\mu(\cdot)$ defined in (\ref{eq:marginal}) should also be well approximated in terms of $T_m$,
as marginals can be written as ratios of partition functions per (\ref{eq:marginal-ratio-part-f}).
Motivated by this we now introduce the definition of pseudo-marginals -- namely values which intend to approximate
marginal values by means of $T_m$. Suppose we are given  a decorated graph $G=(V,E,\mathcal{A})$. 
Consider  Type I or II  interpolation with the interpolating partition function $Z(G(z))$.
In particular,   $Z(G(1))$ is the original partition function $Z(G)$. Recall the definition of $T_m(G,z)$. 
Given a  subset of nodes  $S\subset V$ along with a coloring $\sigma:S\to [K]$, and given an integer $m\ge 0$,
the associated pseudo-marginal $\nu(S,\sigma,m,z)$ is defined as follows. Consider the partition function
$Z(G_{S,\sigma}(z))$ associated with the interpolation of decorated graph $G_{S,\sigma}=(V,E,\mathcal{A}_{S,\sigma})$, 
where $\mathcal{A}_{S,\sigma}$ is defined by (\ref{eq:ASsigma}). 
Let $f(z)=\log Z(G_{S,\sigma}(z))$ and let
\begin{align*}
T_m(G_{S,\sigma},z)=\sum_{0\le k\le m}{1\over k!}f^{(k)}(0)z^k.
\end{align*}
 Recall from (\ref{eq:marginal-ratio-part-f}) that then 
the associated marginals satisfy 
\begin{align*}
\mu(G,S,\sigma)={Z(G,\mathcal{A}_{S,\sigma}) \over Z(G)}.
\end{align*}
The associated pseudo-marginals are defined by 
\begin{align*}
\nu(G,S,\sigma,z,m)={\exp\left(T_m(G_{S,\sigma},z)\right) \over \exp(T_m(G,z))}.
\end{align*}
Similarly, for every $S,T\subset V$ and $\sigma:S\to [K],\tau:T\to [K]$ we define the associated conditional pseudo-marginals as
\begin{align*}
\nu(G,S,\sigma,z,m|T,\tau)&={\nu(G,S\cup T,\sigma\cup T,z,m) \over \nu(G,T,\tau,z,m)} \\
&={\exp\left(T_m(G_{S\cup T,\sigma\cup\tau},z)\right) \over \exp(T_m(G_{T,\tau},z))}.
\end{align*}

The interpretation of pseudo-marginals should be clear. If $T_m(G,z)$ is a good approximation of the log-partition function
$f(z)=\log Z(G(z))$ for large enough $m$, 
then presumably the same
should be true for the reduced log-partition function $\log Z(G_{S,\sigma},z)$, 
obtained when the values of homomorphisms of $\phi$ are fixed to $\sigma(u)$
at $u\in S$. Namely, it should be the case that also 
$T_m(G_{S,\sigma},z)\approx \log Z(G_{S,\sigma}(z))$. In this case we expect to have 
$Z(G(z))\approx \exp(T_m(G,z))$, and
$Z(G_{S,\sigma}(z))\approx \exp(T_m(G_{S,\sigma},z))$,  leading to
\begin{align*}
\mu(G,S,\sigma)  \approx 
{\exp(T_m(G_{S,\sigma},1)) \over \exp(T_m(G,1)))}=\nu(G,S,\sigma,1,m).
\end{align*}
We will prove that the conditional pseudo-marginals $\nu(\cdot|\cdot)$ equal to unconditional pseudo-marginals for sets $S$
when conditioned on a boundary of a sufficiently deep neighborhood $T=\partial B(S,R)$. Namely, the set $S$ and its associated
boundary $\partial B(S,R)$ are ''pseudo-independent''. This is the main technical result of the paper. Then
if the pseudo-marginals provide a good approximation of actual marginals,
the same should apply to marginal distributions in some approximation sense. 
In the remainder of the paper we  write $\nu(G,S,\sigma,m)$ in place of  $\nu(G,S,\sigma,1,m)$ 
and $\nu(G,S,\sigma,m|T,\tau)$ in place of $\nu(G,S,\sigma,1,m|T,\tau)$.

\section{Strong Spatial Mixing. Main result}\label{section:main-result}
In this section we state  our main result: if low-degree Taylor approximation $T_m$ provides a good approximation 
of the log-partition function $\log Z(G)$, 
then the model exhibits a version of the  correlation decay property known as the Strong Spatial Mixing (SSM), which will be defined precisely. 
The main approach is based on showing that  the pseudo-marginals $\nu(\cdot,m)$ associated with sufficiently well separated sets
\emph{always} exhibit 
long range independence property.
Thus if $T_m$ approximates $\log Z(G)$ accurately, then $\nu(\cdot,m)$ approximate accurately the true marginal distributions $\mu(\cdot)$, 
and therefore
the latter have to exhibit long range independence as well, which we prove to be in the form of asymptotic SSM.

We begin by formalizing the notion of SSM. We begin by defining the notion of Spatial Mixing (SM)and then 
observe that due to generality and self-reducibility of our model of decorated graphs, SM implies SSM on appropriately reduced graphs. 
Given a decorated graph $G=(V,E,\mathcal{A})$,  and given any  subset $S\subset [V]$ and positive integer $R$ let
\begin{align}
\rho_R(G,S)&=\max_{\sigma:S\to [K],\tau_1,\tau_2: \partial B(S,R)\to [K]} 
|\mu\left(G,S,\sigma| \partial B(S,R),\tau_1\right)-\mu\left(G,S,\sigma| \partial B(S,R),\tau_2\right)| \notag\\
&=\max_{\sigma:S\to [K],\tau_1,\tau_2: \partial B(S,R)\to [K]} 
|\mu\left(G_{\partial B(S,R),\tau_1},S,\sigma \right)-\mu\left(G_{\partial B(S,R),\tau_2},S,\sigma\right)|. \label{eq:R-range-dependence}
\end{align}
Namely, $\rho_R(G,S)$ denotes the largest sensitivity of the conditional marginal distribution on $S$ with respect to setting the color  values
at the boundary $\partial B(S,R)$. Loosely speaking the model exhibits the SM when $\rho_R(G,S)\approx 0$ for large $R$. Typically,
the case considered in the literature is when the set of interest $S$ is small, often just a singleton. 
Formally, consider a family of decorated graphs $\mathcal{F}$. We say it exhibits the SM if there exists a function $\rho_R^*, R\in Z_+$ which 
converges to zero as $R\to\infty$, such that
\begin{align*}
\max_{G\in\mathcal{F},S\subset V(G)}\rho_R(G,S)\le \rho_R^*.
\end{align*}
In other words $R$-range dependence in the sense of (\ref{eq:R-range-dependence}) decays to zero uniformly in $R$, the  graph  and the
set $S$ choices. 
The SSM property is the SM property which holds when some of the nodes have prescribed colors. Formally, a family of graphs $\mathcal{F}$
exhibits the SSM property if the family of graphs $G_{\Lambda,\nu}$ with  $G\in\mathcal{F}, \Lambda\subset V(G), \eta:\Lambda\to [K]$ 
exhibits the SM property in the sense above. 

In our setting the difference between the SM and the SSM properties is hardly seen, but the difference can be quite substantial. It is known
for example that when $\mathcal{F}$ is the family of all $d$-regular trees, the model exhibits the SM property as soon as 
$K\ge d+1$,\cite{JonassonColoring2002}, whereas for the SSM property it has been established only when $K\ge 1.59d$, \cite{efthymiou2019improved}
and for triangle-free graphs with degree at most $d$ when $K\ge 1.763 d$~\cite{gamarnik2015strong}. 
It is conjectured that it holds as soon as $K\ge d+1$.
The results are essentially equivalent to establishing the SM property 
for the list-coloring problem. The distinction between SM and SSM is also important in structured graphs like lattices. The 
independent set model is known to exhibit the SSM on graphs with degree $d$ 
when $\lambda$ satisfies (\ref{eq:lambda_c}), as was established in~\cite{weitzCounting}, and provably
fails to exhibit the SM on $d$-regular trees, as soon as $\lambda> (d-1)^{d-1}/(d-2)^{d}$,  
which has been known for a while from~\cite{KellyHardCore},\cite{Spitzer75} and~\cite{Zachary83}.

In this paper we consider a weaker asymptotic version of the SSM. We consider a sequence of decorated graphs $G_n=(V_n,E_n,\mathcal{A}_n)$ 
and a sequence of distances $R_n$. We say that this sequence exhibits the asymptotic SSM at distances $R_n$, 
$\lim_{n\to\infty}R_n=\infty$ if
\begin{align*}
\lim_{n\to\infty}\max_{S,\Lambda_n,\subset V_n,\eta_n:\Lambda_n\to [K]}\rho_{R_n}(G_{n,\Lambda_n,\eta_n},S_n)=0.
\end{align*}
The difference of the asymptotic SSM with the SSM property as defined earlier 
is the lack of uniformity of the upper bound on $\rho(\cdot)$ with respect to the 
graphs $G$. To appreciate the distinction, consider the setting when $|V_n|=n$, in the other words the graph has $n$ nodes, and when 
$R_n=C\log^\alpha n$ for some constants $C,\alpha$. Incidentally, this is the setting we consider in our main result with $\alpha=3$. 
Then the asymptotic SSM means long-range independence at distances  $\Omega(\log^\alpha n)$. For some graphs, such as lattices or amenable
graphs in general this is a meaningful property when say $S$ is singleton, 
as the number of nodes further than $O(\log^{\alpha} n)$ distance away from a single node still constitute the bulk of the graph. 
But for some graphs with strong expansion type properties, 
the distances beyond $O(\log n)$ simply might not exist and thus
the property is vacuous. Many lattice models exhibit long range dependence and thus the lack of asymptotic SSM for some choices 
of the parameters. For example the hard-core model on $\Z^2$ exhibits long range dependence when $\lambda>5.3646$ as shown
in Antonio et al~\cite{blanca2013phase}.   The $3$-coloring model exhibits the long range dependence on $\Z^d$ for all sufficiently large $d$~
\cite{galvin2015phase}.
We conjecture that our main result below extends to the case of SSM as originally defined 
and thus to all graphs, including expanders, but we are
not able to prove this yet.

We now state our main technical result.

\begin{theorem}\label{theorem:main-result}
Given a decorated graph $G=(V,E,\mathcal{A})$, consider either the Type I interpolation (associated with the Independent Set model)
or the Type II interpolation.
Then for every $R$,   $S\subset V,\sigma:S\to [K]$ and $\tau:\partial B(S,R)\to [K]$, 
\begin{align*}
\nu(G,S,\sigma,z,(1/2)R|\partial B(S,R),\tau)=\nu(G,S,\sigma,z,R).
\end{align*}
\end{theorem}
In other words, the ''conditional'' pseudo-marginal at $S$ when ''conditioning'' on the boundary of the neighborhood of $S$ at distance $R$
equals to  ''unconditional'' pseudo-marginal, when the pseudo-marginals are computed using the first $R/2$ terms of the associated
Taylor approximation of the log-partition function. As we will see from the proof, the factor $1/2$ does not appear for the Type I interpolation
but does appear for the Type II interpolation.

The implication of this result to the Strong Spatial Mixing property is discussed in the following corollary.

\begin{coro}\label{coro:Main-corollary}
Consider a sequence of graphs $G_n=(V_n,E_n,\mathcal{A}_n)$   
on $n=|V_n|$ nodes, such that $Z(G_n)\ne 0$, for all $n$. 
 Consider either the Type I or Type II interpolation.
Suppose for every $\epsilon>0$ there exists $c(\epsilon)$ such that 
for any sequence $\Lambda_n\subset V_n, \eta_n:\Lambda_n\to [K]$
\begin{align}\label{T_m-approx-II}
(1-\epsilon)Z(G_{n,\Lambda_n,\eta_n})\le \exp\left(T_m(G_{n,\Lambda_n,\eta_n},1)\right)
\le (1+\epsilon)Z(G_{n,\Lambda_n,\eta_n}),
\end{align}
when  $m\ge c(\epsilon)\log n$ and  $n$ is large enough. 
Suppose $R_n=\omega(\log n)$. Then
\begin{align}\label{eq:rho-limit-zero}
\lim_{n\to\infty}\sup_{S_n,\Lambda_n\subset V_n, \eta_n:\Lambda_n\to [K]}\rho_{R_n}(G_{n,\Lambda_n,\eta_n},S_n)=0.
\end{align}
Namely the model exhibits the asymptotic SSM at distances asymptotically larger than $\log n$.
\end{coro}

The result above rules out the possibility of using the interpolation method for models exhibiting the (non-trivial) long-range independence,
including, for example, the independent set model on the 2-dimensional lattice for $\lambda>5.3646$ and $3-coloring$ on the $d$-dimensional lattice,
as discussed earlier.

Let's comment on the assumptions of the theorem, specifically in the context
of concrete models. 
In the case of independent set model, we have trivially $Z(G_n)\ne 0$.   
In the case of graph list-coloring, the equality $Z(G_n)=0$ arises when graph $G_n$ is not list-colorable with the list encoded by $\mathcal{A}_n$,
and it is not a trivial condition to check (in fact it is NP-hard).
Typically though, the interpolation method is established for sequences of graphs and color lists  for which it is easily verified  that the 
partition function is distinct from zero, in part because the method itself is built on identifying a zero free region containing $z=1$.
An example of such assumption is the assumption that the size of each list is larger than the degree of the graph. A stronger assumption than
this was required typically in most papers on approximate counting of colorings, including~\cite{liu2019deterministic}. 

The assumption (\ref{T_m-approx-II}) is just a statement regarding the success of the interpolation method for approximating the partition function.
The  subtlety here  regards the model being reduced by fixing colors of any set $\Lambda_n$. In the context of the independent
set model this amounts to forcing the nodes in $S_n$  to be in or out of the independent set, effectively reducing the underlying graph by 
deleting nodes $u$ in $S_n$ marked $0$ by $\sigma_n$, and deleting nodes and neighbors of $u\in S_n$ marked  $1$ by $\sigma_n$, 
as we have already observed. The degree of the graph
is not increased in this procedure so if the interpolation method was successful for the original graph $G_n$ it presumably should be successful
for the reduced graph sequences as well, since the assumption  regarding successful applications of the interpolation method for independent set model
are typically stated in terms of upper bounds on the graph degree in terms of $\lambda$. We see in particular that (\ref{T_m-approx-II})
holds for the any sequence of degree $d$ bounded graphs and $\lambda$ satisfying (\ref{eq:lambda_cc}) as was established in~\cite{harvey2018computing}.

Similarly, for the case of the problem of counting list-colorings, forcing the colors of $\Lambda_n$ to be ones according to  $\eta_n$ amounts to deleting 
nodes in $\Lambda_n$ and deleting colors $\eta_n(u)$ from the lists associated with neighbors of $u$ in $G_n$. The assumption used in successful
implementation of the interpolation method typically include such  reductions of the graph. Specifically, since this procedure reduces the degree
of each neighbor of $S_n$ and its list color size by the same amount, the typical assumptions which take the form  
''list-size is at least $\alpha$ times the node
degree'', adopted for example in~\cite{liu2019deterministic} is maintained. As mentioned earlier this paper considers
the list-coloring model of triangle-free graphs with list of each node exceeding the degree of each node by a multiplicative factor approximately
$1.763d$. A sequence of graphs satisfying this condition thus satisfies (\ref{T_m-approx-II}) as follows from the result in~\cite{liu2019deterministic}.

We now prove Corollary~\ref{coro:Main-corollary} assuming the validity of Theorem~\ref{theorem:main-result}.

\begin{proof}[Proof of Corollary~\ref{coro:Main-corollary}]
Consider any sequence of graphs $G_n=(V_n,E_n,\mathcal{A}_n)$ satisfying the assumptions
of the theorem. In particular $Z(G_n)>0$. Fix any $\epsilon>0$ and any sequence $S_n,\Lambda_n\subset V_n, \eta_n:\Lambda_n\to [K]$.
We write $G_n$ for $G_{n,\Lambda_n,\eta_n}$ for short. 
Applying (\ref{T_m-approx-II}) and setting $m=c(\epsilon)\log n$
we have for any $\tau:\partial B(S_n,R_n)\to [K]$
\begin{align*}
\mu(G_n,S_n,\sigma_n|\partial B(S_n,R_n),\tau_n)
&={\mu(G_n,S_n\cup\partial B(S_n,R_n),\sigma_n\cup\tau_n) \over  \mu(G_n,\partial B(S_n,R_n),\tau_n)} \\
&={Z(G_{n,S_n\cup\partial B(S_n,R_n),\sigma_n\cup\tau_n}) \over Z(G_{n,\partial B(S_n,R_n),\tau_n})} \\
&\le {1+\epsilon \over 1-\epsilon}{\exp\left(T_m(G_{n,S_n\cup\partial B(S_n,R_n),\sigma_n\cup\tau_n},1)\right) 
\over \exp\left(T_m(G_{n,\partial B(S_n,R_n),\tau_n},1)\right)} \\
&= {1+\epsilon \over 1-\epsilon}\nu\left(G_n,S_n,\sigma_n,m| \partial B(S_n,R_n),\tau_n\right).
\end{align*}

Since $R_n\ge m$ for all sufficiently large $n$, then Aapplying Theorem~\ref{theorem:main-result} the last expression is 
\begin{align*}
{1+\epsilon \over 1-\epsilon}\nu\left(G_n,S_n,\sigma_n,m\right),
\end{align*}
for all large enough $n$.
As $\mu(\cdot)\le 1<(1+\epsilon)/(1-\epsilon)$ we obtain in fact
\begin{align*}
\mu(G_n,S_n,\sigma_n|\partial B(S_n,R_n),\tau_n)\le 
{1+\epsilon \over 1-\epsilon}\min\left(\nu\left(G_n,S_n,\sigma_n,m\right),1\right),
\end{align*}
for all large $n$.
Similarly, we establish that for all enough large $n$
\begin{align*}
\mu(G_n,S_n,\sigma_n|\partial B(S_n,R_n),\tau_n)&\ge 
{1-\epsilon \over 1+\epsilon}\nu\left(G_n,S_n,\sigma_n,m\right) \\
&\ge 
{1-\epsilon \over 1+\epsilon}\min\left(\nu\left(G_n,S_n,\sigma_n,m\right),1\right).
\end{align*}
Considering now two boundary assignments $\tau_{n,1},\tau_{n,2}: \partial B(S_n,R_n)\to [K]$, we obtain 
\begin{align*}
&|\mu(G_n,S_n,\sigma_n|\partial B(S_n,R_n),\tau_{n,1})-\mu(G_n,S_n,\sigma_n|\partial B(S_n,R_n),\tau_{n,2}) \\
&\le \left({1+\epsilon \over 1-\epsilon}-{1-\epsilon \over 1+\epsilon}\right)
\min\left(\nu\left(G_n,S_n,\sigma_n,m\right),1\right) \\
&\le {2\epsilon \over 1-\epsilon^2}.
\end{align*}
As the left-hand side does not depend on $\epsilon$, the result follows.
\end{proof}

\section{Some preliminary results}\label{section:preliminary-results}
In this section we present some simple preliminary results that we need for proving Theorem~\ref{theorem:main-result}
Given a complex variable polynomial $p(z)=c_0(p)+c_1(p)z+\cdots+c_n(p)z^n$ with $c_0(p)$ assumed to be non-zero, 
denote its $n$ non-zero complex roots by $\zeta_1,\ldots,\zeta_n$.
Let $\text{Roots}(p,k)=\sum_{1\le j\le n}\zeta_j^{-k}$. 
The following identity known as Newton identity states 
\begin{align}\label{eq:iterations-roots}
kc_k(p)=-\sum_{i=0}^{k-1}c_i(p) \text{Roots}(p,k-i).
\end{align}
Here $c_k(p)=0$ are assumed for $k>n$.
Its short derivation is given in \cite{patel2017deterministic} and is skipped.
In the special case $c_0(p)=1$ this means that $\text{Roots}(p,k)$ can be computed in terms of $c_1(p),\ldots,c_k(p)$ recursively.
In fact it can be expressed explicitly due to the formula by Girard (developed in fact in 1629, thus before Newton~\cite{tignol2015galois},
see also~\cite{wiki-newton}), via the following relation
\begin{align*}
\text{Roots}(p,k)=(-1)^{k}k\sum _{m_{1}+2m_{2}+\cdots +km_{k}=k \atop m_{1}\geq 0,\ldots ,m_{k}
\geq 0}{\frac {(m_{1}+m_{2}+\cdots +m_{k}-1)!}{m_{1}!m_{2}!\cdots m_{k}!}}\prod _{i=1}^{k}(-c_{i}(p))^{m_{i}}.
\end{align*}
We rewrite this in a more general form
\begin{align}\label{eq:roots-in-c}
\text{Roots}(p,k)=\sum_{m_{1}+2m_{2}+\cdots +km_{k}=k \atop m_{1}\geq 0,\ldots ,m_{k}\geq 0}
\alpha_{m_1,\ldots,m_k}\prod_{1\le i\le  k}c_i^{m_i}(p),
\end{align}
for some coefficients $\alpha_{m_1,\ldots,m_k}$. Considering now $f(z)=\log p(z)=\sum_{1\le i\le n}\log(z-\zeta_i)+\log c_n(p)$
we obtain
\begin{align}
f^{(k)}(0)&=\sum_{1\le i\le n}k!(-1)^{k}\zeta_i^{-k} \notag\\
&=k!(-1)^{k}\text{Roots}(p,k). \label{eq:f-in-roots}
\end{align}
The $m$-order Taylor expansion of $f$ around $z=0$ is then
\begin{align}
T_m(p,z) &\triangleq \sum_{0\le k\le m}{1\over k!}z^k k!(-1)^{k}\text{Roots}(p,k) \notag \\
&=\sum_{0\le k\le m}z^k (-1)^{k}\text{Roots}(p,k)  \label{eq:T-taylor}.
\end{align}

Next we observe the following basic additivity property of the function $\text{Roots}(p,k)$ when $p$ is a interpolated partition function $G(z)$. 
Suppose $G$ 
is a disjoint union of graphs $G_j,j=1,2$. Then by (\ref{eq:partition-factorizes}) 
the set of 
roots of $Z(G(z))$ is the union of  roots of $Z(G_1(z))$ and $Z(G_2(z))$, 
and thus counting multiplicity
\begin{align}\label{eq:roots-additive}
\text{Roots}(Z(G(z)),k)=\text{Roots}(Z(G_1(z)),k)+\text{Roots}(Z(G_2(z)),k).
\end{align}

We now turn to the notion of color-respecting graph isomorphism and color-respecting graph embeddings. 
Given two decorated graphs $F=(V(F),E(F),\mathcal{A}(F))$ and $G=(V(G),E(G),\mathcal{A}(G))$, a mapping $\psi:V(F)\to V(H)$
is a color-respecting graph isomorphism if it is a graph isomorphism with respect to the underlying graphs $(V(F),E(F))$
and $(V(G),E(G))$, if $a^{\psi(u)}(H)=a^u(F)$ for all $u\in V(F)$ and  $A^{(\psi(u),\psi(v))}(H)=A^{(u,v)}(F)$ for all $(u,v)\in E(F)$.
Here $a^u(F),u\in V(F), A^{(u,v)}, (u,v)\in E(F)$ and $a^u(H),u\in V(F), A^{(u,v)}(H), (u,v)\in E(F)$ 
are expanded notations for $\mathcal{A}(F)$ and $\mathcal{A}(H)$, respectively. We have that $A^{(\psi(u),\psi(v))}(H)$ is well defined
for every $(u,v)\in E(F)$ since by the graph isomorphism property $(\psi(u),\psi(v))\in E(H)$. 

Given decorated graphs $F=(V(F),E(F),\mathcal{A}(F))$ and $G=(V(G),E(G),\mathcal{A}(G))$ a mapping $\psi:V(F)\to V(H)$
is a color-respecting embedding if it is a color-respecting graph isomorphism between $F$ and the subgraph of $H$ induced
by the image $\psi(V(F))$. 
We denote by 
$\text{Ind}(F,H)$ the total number of of the subsets of nodes $S\subset V(H)$ such that there exists color respecting graph
isomorphism between $F$ and the decorated subgraph of $H$ induced by $S$. Namely, it is the number of embeddings of $F$ into $H$
up to isomorphism. Later we will use the notation of the form $\sum_{F\in\mathcal{G}_i} \text{Ind}(F,H)$
where the sum is over all uncountable collection $\mathcal{G}_i$, yet it makes sense since only finitely many elements of this collection
have a non-zero value for $\text{Ind}(F,H)$. 

Given a connected decorated graph $F$ and another decorated graph $H$ which is a disjoint union of two decorated
graphs $H_1$ and $H_2$ we  naturally have the following identity 
\begin{align}\label{eq:Sum-Ind-connected}
\text{Ind}(F,H)=\text{Ind}(F,H_1)+\text{Ind}(F,H_2).
\end{align}
The following relation for products of the number of embeddings will be useful. This observation was also used in~\cite{patel2017deterministic}. 

\begin{lemma}\label{lemma:Ind-product}
There exists a sequence of functions $\alpha_m:\mathcal{G}^m\to Z_+$ such that for any decorated graph $H$ and any sequence of decorated graphs 
$F_1,\ldots,F_m$
\begin{align}
\prod_{1\le \ell\le m}\text{Ind}(F_i,H)=\sum\alpha_{m+1}(F_1,\ldots,F_m,F)\text{Ind}(F,H), \label{eq:Ind-product}
\end{align}
where the sum is over $F\in \bar{\mathcal{G}}_t$ with $t=\sum_{1\le \ell\le k}|V(F_\ell)|$.
\end{lemma}

\begin{proof}
For every $m$-tuple of color-respecting isomorphic embeddings $\psi_\ell:V(F_\ell)\to  V(H), 1\le \ell\le m$, consider the subgraph $F$ of $H$ induced by 
the union $\cup_{1\le \ell\le m}\psi(V(F_\ell))$. This graph has at most $\sum_\ell |V(F_\ell)|$ nodes.
We obtain an embedding of this graph $F$ in $H$. Then we see that
(\ref{eq:Ind-product}) holds, where $\alpha_{m+1}(F_1,\ldots,F_m,F)$ is the number of $m$-tuples of embeddings of $F_1,\ldots,F_m$
into $F$ which span $F$. 
\end{proof}
A key property stated in the lemma is that $\alpha_m$ depends on the collection $F_1,\ldots,F_m$ alone and not on the target graph $H$.

\section{Proof of Theorem~\ref{theorem:main-result}}\label{section:main-result-proof}
This section is devoted to the proof of Theorem~\ref{theorem:main-result}.
We prove the claim separately for each interpolation type. Both developments follow ideas similar to ones in~\cite{patel2017deterministic}.
The main distinction is that our development is geared towards establishing the equality between the conditional and unconditional pseudo-marginals,
whereas the goal 
in~\cite{patel2017deterministic} is developing a method of counting connected subgraph in order to obtain a polynomial time algorithm
for computing $T_m(G,z)$. 

\subsubsection*{Type I interpolation}
Fix a  graph $G=(V,E)$, fugacity $\lambda>0$ and consider the associated interpolated partition function (\ref{eq:Partition-hard-core-interpolated})
which we recall here for convenience:
\begin{align*}
Z(G(z))=\sum_{0\le k\le |V|}i_k(G)z^k\lambda^k.
\end{align*}
We note that the free coefficient of this polynomial $i_0=1$. Applying the identity (\ref{eq:roots-in-c}) we have
\begin{align*}
\text{Roots}(Z(G(z)),k)
=\sum_{(m_1,\ldots,m_k)\in \Gamma_k}\alpha_{m_1,\ldots,m_k}\prod_{1\le j\le k}
\left(i_j(G)\lambda^j\right)^{m_j},
\end{align*}
where  $\Gamma_k$ denotes the set of all $m_1,\ldots,m_k\ge 0$ with $\sum_{1\le \ell\le k}\ell m_\ell=k$.
Denote by $I_j$ an independent set of size $j$. Then $i_j(G)$ is $\text{Ind}(I_j,G)$ with respect to trivial coloring $a=A=1$ of both $G$ and $I_j$.
In other words it is the number of isomorphic embeddings of a size $j$ independent set into $G$ purely in graph theoretic sense.
We then rewrite the above as 
\begin{align*}
\text{Roots}(Z(G(z)),k)
=\sum_{(m_1,\ldots,m_k)\in \Gamma_k}\alpha_{m_1,\ldots,m_k}\prod_{1\le j\le k}
\lambda^{jm_j}\left(\text{Ind}(I_j,G)\right)^{m_j}. 
\end{align*}
Expanding the powers $(\cdot)^{m_i}$  and applying Lemma~\ref{lemma:Ind-product}
we see that we can write $\text{Roots}(Z(G(z)),k)$ in the form
\begin{align}\label{eq:roots-k-uncolored}
\text{Roots}(Z(G(z)),k) = 
\sum_{H\in \bar{\mathcal{G}}_{k}} \beta_{H,k}
\text{Ind}(H,G),
\end{align}
The bound $k$ on the size appears since the set spanned by a union of $m_\ell$ copies of $I_\ell$ with $1\le \ell\le k$ has size
at most $k$, in light of $\sum_\ell \ell m_\ell=k$.
A key fact for us is the following lemma.

\begin{lemma}\label{lemma:Zero-disconnected-uncolored}
For every disconnected graph $H$ and every $k$, $\beta_{H,k}=0$.
\end{lemma}

\begin{proof}
This fact  is established in several places including \cite{csikvari2016benjamini} 
and \cite{patel2017deterministic}. We reproduce the proof here for convenience.

Fix any $k$.  Assume for the purposes of contradiction that 
there exists a  disconnected $r$-node  graph $H_0=(V(H_0),E(H_0))$  with $\beta_{H_0,k}\ne 0$.
Without the loss of generality we may assume that $r$ is the smallest value for which such a  graph exists. 
Applying the identity (\ref{eq:roots-k-uncolored}) to $G=H_0$  we have
\begin{align*}
\text{Roots}(Z(H_0(z)),k) &= \sum_{H\in \bar{\mathcal{G}}_{k}} 
\beta_{H,k}\text{Ind}(H,H_0). 
\end{align*}
We expand the right-hand side as

\begin{align}
& \sum_{H_0\ne H\in \bar{\mathcal{G}}_{k}} 
\beta_{H,k}
\text{Ind}(H,H_0) 
+
\beta_{H_0,k}\text{Ind}(H_0,H_0). \label{eq:roots-for-H_0}
\end{align}
We will prove that $\beta_{H_0,k}=0$, thus arriving at contradiction.
Trivially $\text{Ind}(H,H_0)=0$ if $|V(H)|>|V(H_0)|$.
Also $\text{Ind}(H,H_0)=0$ if $|V(H)|=|V(H_0)|$, but $H\ne H_0$ (up to isomorphism). Thus the right-hand
side above is
\begin{align*}
\sum_{H\in \bar{\mathcal{G}}_{k}, |V(H)|<|V(H_0)|} \beta_{H,k}\text{Ind}(H,H_0)+
\beta_{H_0,k}\text{Ind}(H_0,H_0). 
\end{align*}
By the assumption of minimality of $r=|V(H_0)|$ we have $\beta_{H,k}=0$ for all disconnected graphs $H$ with $|V(H)|<|V(H_0)|$. Thus
\begin{align}\label{eq:roots-H-uncolored}
&\text{Roots}(Z(H_0(z)),k) \notag \\
&=\sum_{H\in \bar{\mathcal{G}}_{k,\text{conn}}, |V(H)|<|V(H_0)|} 
\beta_{H,k}\text{Ind}(H,H_0) 
+
\beta_{H_0,k}\text{Ind}(H_0,H_0). 
\end{align}

Let $H_{0,j}, j=1,2$  be any decomposition of $H_0$ into any two disconnected parts. 
For every connected graph $H$ we have by (\ref{eq:Sum-Ind-connected})
\begin{align*}
\text{Ind}(H,H_0)
=\sum_{j=1,2}\text{Ind}(H,H_{0,j}).
\end{align*}
Thus we may rewrite (\ref{eq:roots-H-uncolored}) as
\begin{align}\label{eq:roots-H12-uncolored}
&\text{Roots}(Z(H_0(z)),k) \notag \\
&=\sum_{j=1,2}\sum_{H\in \bar{\mathcal{G}}_{k,\text{conn}}, |V(H)|\le |V(H_{0,j})|} 
\beta_{H,k}\text{Ind}(H,H_{0,j}) 
+
\beta_{H_0,k}\text{Ind}(H_0,H_0). 
\end{align}
Applying (\ref{eq:roots-k-uncolored}) to $H_{0,j}, j=1,2$ we also have
\begin{align*}
\text{Roots}(Z(H_{0,j}(z)),k)  
=
\sum_{H\in \bar{\mathcal{G}}_{k,\text{conn}},|V(H)|\le |V(H_{0,j)}|} \beta_{H,k}
\text{Ind}(H,H_{0,j}).
\end{align*}
By (\ref{eq:roots-additive}) we have
\begin{align*}
\text{Roots}(Z(H_{0}(z)),k)=
\sum_{j=1,2}\text{Roots}(Z(H_{0,j}(z)),k),
\end{align*}
and therefore
\begin{align*}
\text{Roots}(Z(H_{0}(z),k))=
\sum_{j=1,2} \sum_{H\in \bar{\mathcal{G}}_{k,\text{conn}},|V(H)|\le |V(H_{0,j)}|} \beta_{H,k}
\text{Ind}(H,H_{0,j}).
\end{align*}
Comparing  with (\ref{eq:roots-H12-uncolored}) we conclude
\begin{align*}
\beta_{H_0,k}\text{Ind}(H_0,H_0)=0.
\end{align*}
Since $\text{Ind}(H_0,H_0)$  trivially has value at least $1$, 
we conclude $\beta_{H_0,k}=0$
thus arriving at contradiction. 
\end{proof}

Applying (\ref{eq:roots-k-uncolored}) and Lemma~\ref{lemma:Zero-disconnected-uncolored} we have
\begin{align*}
\text{Roots}(Z(G(z)),k) = \sum_{H\in \bar{\mathcal{G}}_{k,\text{conn}}} \beta_{H,k}\text{Ind}(H,G).
\end{align*}
Now let  $f(z)=\log Z(G(z))$.  Applying (\ref{eq:f-in-roots})   we have
\begin{align*}
f^{(k)}(0)=k!(-1)^{k}\sum_{H\in \bar{\mathcal{G}}_{k,\text{conn}}} \beta_{H,k}\text{Ind}(H,G). 
\end{align*}
and from (\ref{eq:T-taylor}) we obtain
\begin{align*}
T_m(G,z)
=\sum_{0\le k\le m}z^k (-1)^{k}\sum_{H\in \bar{\mathcal{G}}_{k,\text{conn}}} \beta_{H,k}\text{Ind}(H,G).
\end{align*}
Similarly, for every $S\subset V$ and $\sigma:S\to [K]$, letting $f(z)=\log Z(G_{S,\sigma}(z))$
we obtain
\begin{align*}
f^{(k)}(0)=k!(-1)^{k}\sum_{H\in \bar{\mathcal{G}}_{k,\text{conn}}} \beta_{H,k}\text{Ind}(H,G_{S,\sigma}),
\end{align*} 
and
\begin{align*}
T_m(G_{S,\sigma},z) 
&=\sum_{0\le k\le m}z^k(-1)^{k}\sum_{H\in \bar{\mathcal{G}}_{k,\text{conn}}} \beta_{H,k}\text{Ind}(H,G_{S,\sigma}).
\end{align*}
We obtain the following representation for the pseudo-marginals:
\begin{align*}
\nu(G,S,\sigma,z,m)={\exp\left(\sum_{0\le k\le m}z^k(-1)^{k}
\sum_{H\in \bar{\mathcal{G}}_{k,\text{conn}}} \beta_{H,k}\text{Ind}(H,G_{S,\sigma})\right) 
\over \exp\left(\sum_{0\le k\le m}z^k(-1)^{k}\sum_{H\in \bar{\mathcal{G}}_{k,\text{conn}}} \beta_{H,k}\text{Ind}(H,G)\right)}.
\end{align*}
Letting $\Delta(H,S,\sigma)=\text{Ind}(H,G)-\text{Ind}(H,G_{S,\sigma})$, this simplifies to
\begin{align*}
\nu(G,S,\sigma,z,m)=
\exp\left(-\sum_{0\le k\le m}z^k(-1)^{k}\sum_{H\in \bar{\mathcal{G}}_{k,\text{conn}}} \beta_{H,k}\Delta(H,S,\sigma)\right).
\end{align*}
Similarly, for any $R$ and the set $S\cup \partial B(S,R)$ with $\tau:\partial B(S,R)\to [K]$ we have
\begin{align*}
&\nu(G,S\cup \partial B(S,R),\sigma\cup \tau,z,m)  \\
&=
\exp\left(-\sum_{0\le k\le m}z^k(-1)^{k}\sum_{H\in \bar{\mathcal{G}}_{k,\text{conn}}} 
\beta_{H,k}\Delta(H,S\cup \partial B(S,R),\sigma\cup\tau)\right),
\end{align*}
and 
\begin{align*}
&\nu(G,\partial B(S,R),\tau,z,m)  \\
&=
\exp\left(-\sum_{0\le k\le m}z^k(-1)^{k}\sum_{H\in \bar{\mathcal{G}}_{k,\text{conn}}} 
\beta_{H,k}\Delta(H,\partial B(S,R),\tau)\right),
\end{align*}

A key observation is that $\Delta(H,S,\sigma)$ involves only copies of  connected graphs $H$ in $G$ with at most $k\le m$ nodes 
which intersect with $S$. 
As a result, when the distance $R$ is sufficiently large the sets of graphs $H$ intersecting $S$ and intersecting $\partial B(S,R)$ 
are disjoint. Specifically, if $R\ge  m$ then for every $H$ with $V(H)\cap S\ne\emptyset$ we have $V(H)\cap \partial B(S,R)=\emptyset$,
and vice verse. As a result
\begin{align*}
\Delta(H,S\cup \partial B(S,R),\sigma\cup\tau)=
\Delta(H,S,\sigma)+\Delta(H,\partial B(S,R),\tau).
\end{align*}
Therefore,
\begin{align*}
\nu(G,S,\sigma,z,R|\partial B(S,R),\tau)
&={\nu(G,S\cup \partial B(S,R),\sigma\cup \tau,z,R) \over \nu(G,\partial B(S,R),\tau,z,R)} \\
&=\exp\left(-\sum_{0\le k\le R}z^k(-1)^{k}\sum_{H\in \bar{\mathcal{G}}_{k,\text{conn}}} 
\beta_{H,k}\Delta(H,S,\sigma)\right) \\
&=\nu(G,S,\sigma,z,R).
\end{align*}
This completes the proof of the theorem for the case of Type I interpolation.

\subsubsection*{Type II interpolation}
Turning next to the Type II interpolation, fix a decorated graph $G=(V,E,\mathcal{A})$ with the decoration 
$\mathcal{A}=(a^u, u\in V, A^{(u,v)}, (u,v)\in E)$.
Recall the definition of $L$ from (\ref{eq:LGA})  and consider
the associated renormalized polynomial
\begin{align*}
\bar Z(G(z))& \triangleq L^{-1} Z(G(z))\\
&=L^{-1}\sum_{\phi:V(G)\to [K]}\prod_{u\in V(G)}a^u_{\phi(u)}\prod_{(u,v)\in E(G)}\left(1+z\left(A^{(u,v)}_{\phi(u),\phi(v)}-1\right)\right).
\end{align*}
By construction $\bar Z(G(0))=1$. Introduce a modified decoration $\bar{\mathcal{A}}$ of the underlying graph $(V,E)$ as follows:
\begin{align}
\bar a^u&={a^u\over \sum_{i\in [K]}a^u_i}, \qquad u\in V, \label{eq:a-reweighted}\\
\bar A^{(u,v)}&=A^{(u,v)}-1, \qquad (u,v)\in E. \label{eq:A-reweighted}
\end{align} 
We have
\begin{align}\label{eq:au-normalized}
\sum_{i\in [K]} \bar a^u_i=1, \qquad \forall~ u\in V.
\end{align}
Denote by $\bar G$ the graph $(V,E)$ with this modified decoration $\bar{\mathcal{A}}$.
For any decorated graph $H=(V(H),E(H),\mathcal{A}(H))\in\mathcal{G}_{i,\text{edge}}$ let
\begin{align*}
Z_i(H)=\sum_{E'}\sum_{\phi:V(H)\to [K]} \prod_{u\in V(H)} \bar a^{H,u}_i
\prod_{(u,v)\in E'}\bar A^{H,(u,v)}_{\phi(u),\phi(v)},
\end{align*}
where the outer sum is taken over all subsets of edges $E'\subset E(H)$ which span $H$ and which have cardinality $|E'|=i$. 
Here $a^{H,\cdot}_{\cdot}$ and $A^{H,\cdot}_{\cdot}$ are the decorations associated with $\mathcal{A}(H)$, and
the bar operation is defined for the decoration $\mathcal{A}(H)$ as per (\ref{eq:a-reweighted}) and (\ref{eq:A-reweighted}).
$Z_i(H)$ is a partition function type object except the products over edges are taken only over spanning subsets of the edges
of $H$ with cardinality exactly $i$.

Expanding the product 
\begin{align*}
\prod_{(u,v)\in E(G)}\left(1+z\left(A^{(u,v)}_{\phi(u),\phi(v)}-1\right)\right)
\end{align*} 
in powers of $z$, we claim  that the following representation holds:

\begin{lemma}\label{lemma:Z-power-series}
\begin{align*}
\bar Z(G(z))=\sum_{0\le i\le |V|}z^i\sum_{H\in \mathcal{G}_{i,\text{edge}}} Z_i(H)\text{Ind}(H,\bar G).
\end{align*}
\end{lemma}
As noted earlier, the same graph $H$ may appear in summands corresponding to more than one values of $i$, as the graph can 
be spanned by different number of edges. The contribution to the $\bar Z(G(z))$ though is different for different values of $i$
as those will correspond to different powers of $z$.

\begin{proof}
 The coefficient associated with $z^i$ in polynomial $\bar Z(G(z))$ is 
\begin{align}\label{eq:sum-phi}
\sum_{\phi:V(G)\to [K]}\sum_{E'\subset E: |E|=i} L^{-1}\prod_{u\in V(G)}a^u_{\phi(u)}\prod_{(u,v)\in E'}\left(A^{(u,v)}_{\phi(u),\phi(v)}-1\right).
\end{align}
Denote by $H(E')\in \mathcal{G}_{i,\text{edge}}$ the subgraph of $G$ spanned by edges in $E'$. Then  
the sum in (\ref{eq:sum-phi})  is 
\begin{align*}
&=\sum_{E'\subset E: |E|=i}\sum_{\phi:V(G)\to [K]} L^{-1}\prod_{u\in V(H(E'))}a^u_i
\prod_{(u,v)\in E'}\left(A^{(u,v)}_{\phi(u),\phi(v)}-1\right)
\prod_{u\notin V(H(E'))}a^u_{\phi(u)} \\
&=\sum_{E'\subset E: |E|=i}\left(\sum_{\phi:V(H(E'))\to [K]} L^{-1}\prod_{u\in V(H(E'))} a^u_i
\prod_{(u,v)\in E'}\left(A^{(u,v)}_{\phi(u),\phi(v)}-1\right)\right) \times \\
&\times \left(\sum_{\phi:V\setminus V(H(E'))\to [K]}
\prod_{u\in V\setminus V(H(E'))}a^u_{\phi(u)}\right) \\
&=\sum_{E'\subset E: |E|=i}\left(\sum_{\phi:V(H(E'))\to [K]} \prod_{u\in V(H(E'))} \bar a^u_i
\prod_{(u,v)\in E'}\bar A^{(u,v)}_{\phi(u),\phi(v)}\right),
\end{align*}
Here in the second equality the map $\phi:V(G)\to [K]$ is partition into its reduction to $V(H(E'))$ and its complement, and the product
form structure is used. The last equality follows from the definition of $L$ and $\bar a^u$. 
We recognize the last expression as 
\begin{align*}
\sum_{H\in \mathcal{G}_{i,\text{edge}}} Z_i(H)\text{Ind}(H,\bar G).
\end{align*}
\end{proof}

Using the representation (\ref{eq:roots-in-c}) for the polynomial $p(z)=\bar Z(G(z))$ and since the roots of $\bar Z(G(z))$ 
and $Z(G(z))$ are identical,
we obtain
\begin{align*}
&\text{Roots}(Z(G(z)),k)\\
&=\sum_{(m_1,\ldots,m_k)\in\Gamma_k}\alpha_{m_1,\ldots,m_k}\prod_{1\le i\le k}
\left(\sum_{H\in \mathcal{G}_{i,\text{edge}}} Z_i(H)
\text{Ind}(H,\bar G)\right)^{m_i}. 
\end{align*}
Next we expand the powers $(\cdot)^{m_i}$. Each graph $H\in \mathcal{G}_{i,\text{edge}}$ has at most $2i$ nodes.
Applying Lemma~\ref{lemma:Ind-product}, and using $\sum_\ell \ell m_\ell=k$ for each $(m_1,\ldots,m_k)\in\Gamma_k$
we obtain a representation for every $k$ of the form:
\begin{align}\label{eq:Roots-in-beta}
\text{Roots}(Z(G(z)),k) = \sum_{H\in \bar{\mathcal{G}}_{2k}} \beta_{H,k}
\text{Ind}(H,\bar G),
\end{align}
where $\beta_{H,k}$ depend on the decorated graph $H$ and $k$ only. Note that by (\ref{eq:au-normalized}) we must have  $\beta_{H,k}=0$ unless
$\mathcal{A}(H)$ satisfies 
\begin{align}\label{eq:au-in-H-normalized}
\sum_{i\in [K]}a^u_i(H)=1, \qquad u\in V(H).
\end{align}

\begin{lemma}\label{lemma:Zero-disconnected}
For every disconnected graph $H$ and every $k$, $\beta_{H,k}=0$.
\end{lemma}

\begin{proof}
The proof is similar to the one of Lemma~\ref{lemma:Zero-disconnected-uncolored}, but with a minor adaptation required to handle the case
of decorated graph. A similar property for decorated graphs is also found in \cite{patel2017deterministic} for a different notion of 
color respecting isomorphisms.

Fix any $k$.  Assume for the purposes of contradiction that 
there exists a  disconnected $r$-node decorated graph $H_0=(V(H_0),E(H_0),\mathcal{A}(H_0))$  with $\beta_{H_0,k}\ne 0$.
Without the loss of generality we may assume that $r$ is the smallest value for which such a decorated graph exists.
Let us construct a coloring $\mathcal{A}_1$ of $H_0$ such that $\bar{\mathcal{A}}_1=\mathcal{A}(H_0)$, where the
transformation $\mathcal{A}_1\to \bar{\mathcal{A}}_1$ is obtained by (\ref{eq:a-reweighted}) and (\ref{eq:A-reweighted}). 
This is achieved by simply adding $1$ to every value
$A^{(u,v)}_{i,j}(H_0), (u,v)\in E(H_0), 1\le i,j\le K$, and leaving $a^u(H_0), u\in V(H_0)$ intact, due to (\ref{eq:au-in-H-normalized}).
The graph $(V(H_0),E(H_0)$ with this new coloring $\mathcal{A}_1(H_0)$ is denoted by $H_0'$. 
Applying the identity (\ref{eq:Roots-in-beta}) to $G=H_0'$  we have
\begin{align*}
\text{Roots}(Z(H_0'(z)),k) &= \sum_{H\in \bar{\mathcal{G}}_{2k}} 
\beta_{H,k}\text{Ind}(H,\bar H_0') \notag\\
&=\sum_{H\in \bar{\mathcal{G}}_{2k}} \beta_{H,k}\text{Ind}(H,H_0), 
\end{align*}
where the second equality is obtained since $\bar H_0'=H_0$. We expand the right-hand side as

\begin{align}
& \sum_{H_0\ne H\in \bar{\mathcal{G}}_{2k}} 
\beta_{H,k}
\text{Ind}(H,H_0) 
+
\beta_{H_0,k}\text{Ind}(H_0,H_0). \label{eq:roots-for-H_0}
\end{align}
We will prove that $\beta_{H_0,k}=0$, thus arriving at contradiction.
Trivially $\text{Ind}(H,H_0)=0$ if $|V(H)|>|V(H_0)|$.
Also $\text{Ind}(H,H_0)=0$ if $|V(H)|=|V(H_0)|$, but $H\ne H_0$ (up to isomorphism). Thus the right-hand
side above is
\begin{align*}
\sum_{H\in \bar{\mathcal{G}}_{2k}, |V(H)|<|V(H_0)|} \beta_{H,k}\text{Ind}(H,H_0)+
\beta_{H_0,k}\text{Ind}(H_0,H_0). 
\end{align*}
By the assumption of minimality of $r=|V(H_0)|$ we have $\beta_{H,k}=0$ for all disconnected graphs $H$ with $|V(H)|<|V(H_0)|$. Thus
\begin{align}\label{eq:roots-H}
&\text{Roots}(Z(H_0'(z),k) \notag \\
&=\sum_{H\in \bar{\mathcal{G}}_{2k,\text{conn}}, |V(H)|<|V(H_0)|} 
\beta_{H,k}\text{Ind}(H,H_0) 
+
\beta_{H_0,k}\text{Ind}(H_0,H_0). 
\end{align}

Let $H_{0,j}, j=1,2$  be any decomposition of $H_0$ into any two disconnected parts, with respective coloring reductions
$\mathcal{A}(H_{0,j}), j=1,2$. We denote by  
 $H_{0,j}', j=1,2$  the same decomposition but with respect to coloring $\mathcal{A}_1$.
For every connected graph $H$ we have by (\ref{eq:Sum-Ind-connected}). 
\begin{align*}
\text{Ind}(H,H_0)
=\sum_{j=1,2}\text{Ind}(H,H_{0,j}).
\end{align*}
Thus we may rewrite (\ref{eq:roots-H}) as
\begin{align}\label{eq:roots-H12}
&\text{Roots}(Z(H_0'(z),k) \notag \\
&=\sum_{j=1,2}\sum_{H\in \bar{\mathcal{G}}_{2k,\text{conn}}, |V(H)|\le |V(H_{0,j})|} 
\beta_{H,k}\text{Ind}(H,H_{0,j}) 
+
\beta_{H_0,k}\text{Ind}(H_0,H_0). 
\end{align}
Applying (\ref{eq:Roots-in-beta}) for $H_{0,j}', j=1,2$ we also have
\begin{align*}
\text{Roots}(Z(H_{0,j}'(z),k)  
=
\sum_{H\in \bar{\mathcal{G}}_{2k,\text{conn}},|V(H)|\le |V(H_{0,j)}|} \beta_{H,k}
\text{Ind}(H,H_{0,j}).
\end{align*}
By (\ref{eq:roots-additive}) we have
\begin{align*}
\text{Roots}(Z(H_{0}'(z),k)=
\sum_{j=1,2}\text{Roots}(Z(H_{0,j}'(z),k),
\end{align*}
and therefore
\begin{align*}
\text{Roots}(Z(H_{0}'(z),k)=
\sum_{j=1,2} \sum_{H\in \bar{\mathcal{G}}_{2k,\text{conn}},|V(H)|\le |V(H_{0,j)}|} \beta_{H,k}
\text{Ind}(H,H_{0,j}).
\end{align*}
Comparing  with (\ref{eq:roots-H12}) we conclude
\begin{align*}
\beta_{H_0,k}\text{Ind}(H_0,H_0)=0.
\end{align*}
Since $\text{Ind}(H_0,H_0)$  trivially has value at least $1$, 
we conclude that $\beta_{H_0,k}=0$,
thus arriving at contradiction. 
\end{proof}

Applying (\ref{eq:Roots-in-beta}) and Lemma~\ref{lemma:Zero-disconnected} we have
\begin{align*}
\text{Roots}(Z(G(z)),k) = \sum_{H\in \bar{\mathcal{G}}_{2k,\text{conn}}} \beta_{H,k}\text{Ind}(H,\bar G).
\end{align*}
The remainder of the proof is the same as for the case of Type I interpolation except that the value of $R$
has to change to $2k$ as opposed to $k$, since the sum is over connected graphs with at most $2k$ nodes.

\section*{Acknowledgement}
Several insightful conversations with Alexander Barvinok are gratefully acknowledged. The author
is  very grateful to Guus Regts for suggesting more up to date references and for a suggestion
on improving the bound appearing in the main result, Corollary~\ref{coro:Main-corollary}. 
Many thanks Yuzhou Gu and Yuri Polyanski for making the author aware of the Girard's formula.

\bibliographystyle{amsalpha}

\newcommand{\etalchar}[1]{$^{#1}$}
\providecommand{\bysame}{\leavevmode\hbox to3em{\hrulefill}\thinspace}
\providecommand{\MR}{\relax\ifhmode\unskip\space\fi MR }
\providecommand{\MRhref}[2]{%
  \href{http://www.ams.org/mathscinet-getitem?mr=#1}{#2}
}
\providecommand{\href}[2]{#2}


\begin{thebibliography}{BGK{\etalchar{+}}07}

\bibitem[Bar15]{barvinok2015computing}
Alexander Barvinok, \emph{Computing the partition function for cliques in a
  graph}, Theory OF Computing \textbf{11} (2015), no.~13, 339--355.

\bibitem[Bar16]{barvinok2016computing}
\bysame, \emph{Computing the permanent of (some) complex matrices}, Foundations
  of Computational Mathematics \textbf{16} (2016), no.~2, 329--342.

\bibitem[Bar17a]{barvinok2017combinatorics}
\bysame, \emph{Combinatorics and complexity of partition functions}, Algorithms
  and Combinatorics \textbf{30} (2017).

\bibitem[Bar17b]{barvinok2017computing}
\bysame, \emph{Computing the partition function of a polynomial on the boolean
  cube}, A Journey Through Discrete Mathematics, Springer, 2017, pp.~135--164.

\bibitem[Bar19]{barvinok2019approximating}
\bysame, \emph{Approximating real-rooted and stable polynomials, with
  combinatorial applications}, OnIine journal of analytic combinatorics (2019),
  no.~14.

\bibitem[BG06]{BandyopadhyayGamarnikCountingConference}
A.~Bandyopadhyay and D.~Gamarnik, \emph{Counting without sampling. {N}ew
  algorithms for enumeration problems using statistical physics.}, Proceedings
  of 17th ACM-SIAM Symposium on Discrete Algorithms (SODA), 2006.

\bibitem[BG08]{BandyopadhyayGamarnikCounting}
\bysame, \emph{Counting without sampling. {A}symptotics of the log-partition
  function for certain statistical physics models}, Random Structures and
  Algorithms \textbf{33} (2008), no.~4, 452--479.

\bibitem[BGK{\etalchar{+}}07]{BayatiGamarnikKatzNairTetali}
M.~Bayati, D.~Gamarnik, D.~Katz, C.~Nair, and P.~Tetali, \emph{Simple
  deterministic approximation algorithms for counting matchings}, Proc. 39th
  Ann. Symposium on the Theory of Computing (STOC), 2007.

\bibitem[BGRT13]{blanca2013phase}
Antonio Blanca, David Galvin, Dana Randall, and Prasad Tetali, \emph{Phase
  coexistence and slow mixing for the hard-core model on ? 2}, Approximation,
  Randomization, and Combinatorial Optimization. Algorithms and Techniques,
  Springer, 2013, pp.~379--394.

\bibitem[CF16]{csikvari2016benjamini}
P{\'e}ter Csikv{\'a}ri and P{\'e}ter~E Frenkel, \emph{Benjamini--schramm
  continuity of root moments of graph polynomials}, European Journal of
  Combinatorics \textbf{52} (2016), 302--320.

\bibitem[DS87]{dobrushin1987completely}
Roland~L Dobrushin and Senya~B Shlosman, \emph{Completely analytical
  interactions: constructive description}, Journal of Statistical Physics
  \textbf{46} (1987), no.~5-6, 983--1014.

\bibitem[EGH{\etalchar{+}}19]{efthymiou2019improved}
Charilaos Efthymiou, Andreas Galanis, Thomas~P Hayes, Daniel Stefankovic, and
  Eric Vigoda, \emph{Improved strong spatial mixing for colorings on trees},
  arXiv preprint arXiv:1909.07059 (2019).

\bibitem[Geo88]{GeorgyGibbsMeasure}
H.~O. Georgii, \emph{Gibbs measures and phase transitions}, de Gruyter Studies
  in Mathematics 9, Walter de Gruyter \& Co., Berlin, 1988.

\bibitem[GK12]{GamarnikKatz}
D.~Gamarnik and D.~Katz, \emph{Correlation decay and deterministic {FPTAS} for
  counting list-colorings of a graph}, Journal of Discrete Algorithms
  \textbf{12} (2012), 29--47.

\bibitem[GKM15]{gamarnik2015strong}
David Gamarnik, Dmitriy Katz, and Sidhant Misra, \emph{Strong spatial mixing of
  list coloring of graphs}, Random Structures \& Algorithms \textbf{46} (2015),
  no.~4, 599--613.

\bibitem[GKRS15]{galvin2015phase}
David Galvin, Jeff Kahn, Dana Randall, and Gregory~B Sorkin, \emph{Phase
  coexistence and torpid mixing in the 3-coloring model on
  $\backslash$mathbbz\^{}d}, SIAM Journal on Discrete Mathematics \textbf{29}
  (2015), no.~3, 1223--1244.

\bibitem[HPR20]{helmuth2020algorithmic}
Tyler Helmuth, Will Perkins, and Guus Regts, \emph{Algorithmic pirogov--sinai
  theory}, Probability Theory and Related Fields \textbf{176} (2020), no.~3,
  851--895.

\bibitem[HSV18]{harvey2018computing}
Nicholas~JA Harvey, Piyush Srivastava, and Jan Vondr{\'a}k, \emph{Computing the
  independence polynomial: from the tree threshold down to the roots},
  Proceedings of the Twenty-Ninth Annual ACM-SIAM Symposium on Discrete
  Algorithms, SIAM, 2018, pp.~1557--1576.

\bibitem[Jer03]{jerrum2003counting}
Mark Jerrum, \emph{Counting, sampling and integrating: algorithms and
  complexity}, Springer Science \& Business Media, 2003.

\bibitem[JKP20]{jenssen2020algorithms}
Matthew Jenssen, Peter Keevash, and Will Perkins, \emph{Algorithms for\#
  bis-hard problems on expander graphs}, SIAM Journal on Computing \textbf{49}
  (2020), no.~4, 681--710.

\bibitem[Jon02]{JonassonColoring2002}
J.~Jonasson, \emph{Uniqueness of uniform random colorings of regular trees},
  Statistics and Probability Letters \textbf{57} (2002), 243--248.

\bibitem[JS97]{JerrumSinclairHochbaumApproxAlgorithms}
M.~Jerrum and A.~Sinclair, \emph{The {M}arkov chain {M}onte {C}arlo method: an
  approach to approximate counting and integration}, Approximation algorithms
  for {NP}-hard problems (D.~Hochbaum, ed.), PWS Publishing Company, Boston,
  MA, 1997.

\bibitem[Kel85]{KellyHardCore}
F.~Kelly, \emph{Stochastic models of computer communication systems}, J. R.
  Statist. Soc. B \textbf{47} (1985), no.~3, 379--395.

\bibitem[LLY13]{li2013correlation}
Liang Li, Pinyan Lu, and Yitong Yin, \emph{Correlation decay up to uniqueness
  in spin systems}, Proceedings of the Twenty-Fourth Annual ACM-SIAM Symposium
  on Discrete Algorithms, Society for Industrial and Applied Mathematics, 2013,
  pp.~67--84.

\bibitem[LSS19]{liu2019deterministic}
Jingcheng Liu, Alistair Sinclair, and Piyush Srivastava, \emph{A deterministic
  algorithm for counting colorings with 2-delta colors}, 2019 IEEE 60th Annual
  Symposium on Foundations of Computer Science (FOCS), IEEE, 2019,
  pp.~1380--1404.

\bibitem[LY13]{lu2013improved}
Pinyan Lu and Yitong Yin, \emph{Improved fptas for multi-spin systems},
  Approximation, Randomization, and Combinatorial Optimization. Algorithms and
  Techniques, Springer, 2013, pp.~639--654.

\bibitem[PR17]{patel2017deterministic}
Viresh Patel and Guus Regts, \emph{Deterministic polynomial-time approximation
  algorithms for partition functions and graph polynomials}, SIAM Journal on
  Computing \textbf{46} (2017), no.~6, 1893--1919.

\bibitem[PR{\etalchar{+}}19]{peters2019conjecture}
Han Peters, Guus Regts, et~al., \emph{On a conjecture of sokal concerning roots
  of the independence polynomial}, The Michigan Mathematical Journal
  \textbf{68} (2019), no.~1, 33--55.

\bibitem[Sly10]{sly2010computational}
Allan Sly, \emph{Computational transition at the uniqueness threshold},
  Foundations of Computer Science (FOCS), 2010 51st Annual IEEE Symposium on,
  IEEE, 2010, pp.~287--296.

\bibitem[Spi75]{Spitzer75}
F.~Spitzer, \emph{Markov random fields on an infinite tree}, Ann. Prob.
  \textbf{3} (1975), 387--398.

\bibitem[Tig15]{tignol2015galois}
Jean-Pierre Tignol, \emph{Galois' theory of algebraic equations}, World
  Scientific Publishing Company, 2015.

\bibitem[Wei06]{weitzCounting}
D.~Weitz, \emph{Counting independent sets up to the tree threshold}, Proc. 38th
  Ann. Symposium on the Theory of Computing, 2006.

\bibitem[wik]{wiki-newton}
\emph{https://en.wikipedia.org/wiki/newton's\_identities}.

\bibitem[Zac83]{Zachary83}
S.~Zachary, \emph{Countable state space {M}arkov random fields and {M}arkov
  chains on tree}, Ann. Prob. \textbf{11} (1983), 894--903.

\end{thebibliography}

\end{document}